\documentclass[preprint,pdftex]{imsart}
\usepackage{amsthm,amsmath,amssymb,amsfonts,epic,latexsym,multirow,ulem,bbm,dsfont}
\RequirePackage{natbib}
\RequirePackage[colorlinks,citecolor=blue,urlcolor=blue]{hyperref}
\usepackage[dvips]{graphicx}
\usepackage[T1]{fontenc} 
\usepackage[utf8]{inputenc}

\startlocaldefs

\newtheorem{theo}{Theorem}[section]
\newtheorem{defi}[theo]{Definition}
\newtheorem{prop}[theo]{Proposition}

\newtheorem{assu}{Assumption}

\newtheorem{exam}{Example}

\newtheorem{lemm}[theo]{Lemma}
\newtheorem{rema}[theo]{Remark}

\def\R{\mathbb{R}}
\def \N{\mathbb{N}}
\def \Z{\mathbb{Z}}
\def \A{\mathcal{A}}

\newcommand{\sem}[1]{\mbox{$[\![ #1 ]\!]$}}

\def\1{\mathbbm{1}}
\def \P{\mathbb{P}} 
\def \E{\mathbb{E}} 
\def \F{{\cal F}} 
\def \B{{\cal B}} 
\def \Pstar{\mathbb{P}^\star}
 
\def \Et{\mathbb{E}^\theta}

\def \bEt{\overline{\mathbb{E}}^\theta}

\def \bEs{\overline{\mathbb{E}}^{\star}}
\def \Pt{\mathbb{P}^\theta}
\def \Pts{\mathbb{P}^{\ts}}
\def \bPt{\overline{\mathbb{P}}^\theta}

\def \bPs{\overline{\mathbb{P}}^{\star}}
\def \bPtZ{\overline{\mathbb{P}}^{\theta,Z}}

\newcommand{\ts}{{\theta^{\star}}}

\newcommand{\PPt}{\mathbf{P}^{\theta}}
\newcommand{\EEt}{\mathbf{E}^{\theta}}

\newcommand{\argmax}{\mathop{\rm Argmax}}

\def \eps{\varepsilon}

\def\w{\omega}
\def\bw{{\boldsymbol \omega}}

\def\bX{\mathbf X}
\def\bXn{\mathbf{X}_n}

\def\t{\theta}

\endlocaldefs

\begin{document}

\begin{frontmatter}


\title{Hidden Markov model for parameter estimation of a random walk in a Markov environment. }

\runtitle{HMM for random walk in Markov environment}

\author{\fnms{Pierre} \snm{Andreoletti}\ead[label=e1]{Pierre.Andreoletti@univ-orleans.fr}}
\address{
Laboratoire   MAPMO,   UMR~CNRS~6628,  F\'ed\'eration   Denis-Poisson,
Universit\'e d’Orl\'eans, Orl\'eans, France. \printead{e1}}
\and 
\author{\fnms{Dasha} \snm{Loukianova}\ead[label=e2]{dasha.loukianova@maths.univ-evry.fr}}
\address{Laboratoire de Math\'ematiques et Mod\'elisation d'\'Evry,
Universit\'e   d'\'Evry  Val  d'Essonne,   UMR~CNRS~8071,  \'Evry,
France. \printead{e2}}
\and
\author{\fnms{Catherine} \snm{Matias}\corref{}\ead[label=e3]{catherine.matias@math.cnrs.fr}}
\address{Laboratoire de  Probabilit\'es   et    Mod\`eles   Al\'eatoires,
UMR~CNRS~7599, Universit\'e Pierre  et Marie Curie, Universit\'e Paris
Diderot, Paris, France. \printead{e3}}

\runauthor{Andreoletti et al.}

\begin{abstract}
We focus on the parametric  estimation of the distribution of a Markov
environment from the  observation of a single trajectory of
a  one-dimensional  nearest-neighbor  path  evolving  in  this  random
environment.  In the ballistic case, as the length of the path increases, we prove
consistency,  asymptotic  normality  and  efficiency  of  the  maximum
likelihood  estimator.  Our  contribution  is two-fold:  we  cast  the
problem into the one of parameter estimation in a hidden Markov model (HMM)
and establish that  the bivariate Markov chain underlying  this HMM is
positive Harris recurrent. We  provide different examples of setups in
which our  results apply, in particular that of DNA unzipping model, and we give  a simple  synthetic experiment to
illustrate those results. 
\end{abstract}




\end{frontmatter}

\section{Introduction}

Random walks in random environments (RWRE) form a subclass of canonical models in the 
more general framework of random motions in random media that is widely
used in physics. 
 These models go back to the pioneer works of~\cite{Chernov}, who introduced them to
describe DNA replication and of~\cite{Temkin} who used them in the field
of metallurgy. A more complete list of application fields may be found
in the introduction  of \cite{Bogachev} as well as  in the references
therein. These models  have been intensively studied in  the last four
decades,    mostly   in   the    physics   and    probability   theory
literature. Some surveys on the topic include~\cite{Hughes,Shi,Zeitouni,Revesz}.\\

Statistical issues  raised by those processes have  been overlooked in
the literature  until very  recently, when new  biophysics experiments
produced data that can be  modeled (at least in an ideal-case setup)
by RWRE~\citep{Balda_06}. Consequently, a  new series of works appeared
on statistical  procedures aiming  at estimating parameters  from RWRE
data.  Another motivation to these studies comes from the fact that 
 RWRE  are  closely linked  to
branching  processes with immigration  in random  environments (BPIRE)
and that the statistical issues  raised in one context may potentially
have an impact in the other.

\cite{andreo_sinai} investigates the local
  time of the one dimensional  recurrent RWRE in order to estimate the
  trajectories of the underlying random potential. 
 In~\cite{Andreo},   the ideal-case model from~\cite{Balda_06} is considered: a (finite length) DNA molecule is unzipped several
times and  some device translates  these unzippings into  random walks
along  the  DNA  molecule,  whose  sequence  of  bases  is  the  random
environment.  Here, the goal  is to  reconstruct this  environment and
thus achieve sequencing of the molecule. \cite{Andreo} prove that a Bayesian estimator (maximum a posteriori)
of this  sequence of bases is  consistent as the  number of unzippings
increases  and  they  characterize  the probability  of  reconstruction
error. 
In  a   different  setting,   several  authors  have   considered  the
information on the environment that is contained in one single trajectory of
the walk  with infinite length.  In their pioneer work, \cite{AdEn}
consider  a  very  general  RWRE  and provide  equations  linking  the
distribution of some  statistics of the trajectory to  some moments of
the  environment   distribution.   In the specific case of a
one-dimensional  nearest neighbor path,  those equations  give moment
estimators for the environment distribution parameters.
 More recently, \cite{Comets_etal} studied a maximum likelihood estimator (MLE)
in the specific context of a one-dimensional nearest neighbor path in
transient ballistic regime. They prove the
consistency  of  this  estimator  (as  the length  of  the  trajectory
increases). From a  numerical point of view, this  MLE outperforms the
moment estimator constructed from the work of \cite{AdEn}.  In a 
companion  article~\citep{Falc_etal},  they  have further  studied  the
asymptotic  normality of  the  MLE (still  in  the ballistic  regime),
showed its asymptotic efficiency and
constructed confidence  intervals for  the parameters.  This  work has
been    extended    to     the    transient    sub-ballistic    regime
in~\cite{Falc_bis}.  In this body of work on maximum
  likelihood procedures, the results rely on the
branching  structure of  the  sequence  of the  number  of left  steps
performed by the walk, which was originally observed by \cite{KKS}. 
In the recurrent  case, as the walk visits  every sites infinitely
  often, this  branching process of  left steps explodes and  the same
  approach is useless there. In theory, it is possible in this case to
  estimate the environment itself at each site, and then show that the
  empirical measure  converges to  its distribution. The  problem with
  such a  "naive" approach is the localization  phenomena of recurrent
  RWRE,  discovered by  \cite{Sinai}: most  of the
  sites visited by the RWRE will be extremely few visited, because the
  walk  spends  a  majority  of   its  time  in  the  valleys  of  the
  potential \citep{Andreoletti,Andreoletti2}. This non uniformity is automatically handled with the 
  approach followed by \cite{Comets_bis} and the authors establish
  consistency of two estimators, a MLE and a maximum pseudo-likelihood
  estimator.  \\

We now stress  the fact that all the  previously mentioned statistical
works but the one from \cite{Andreo} are valid in the case of an environment composed
of   independent   and   identically   distributed   (i.i.d.)   random
variables. While very convenient, this assumption might be restrictive
in some contexts, e.g. DNA modeling. 
In the present work, we investigate the statistical estimation of a
parametric  Markov   environment  from   a  single  trajectory   of  a
one-dimensional  nearest-neighbor path, when  its length  increases to
infinity. We  consider the case of  a transient walk  in the ballistic
regime. Our contribution is twofold: first, we show how the problem is
cast into  the one  of estimating the  parameter of a  hidden Markov
model (HMM), or  more specifically  of a first-order  autoregressive process
with Markov regime. Indeed, the RWRE itself is not a HMM but the branching process of the sequence of
  left steps performed by the walk, is.
 Second, we prove that the bivariate Markov chain
that   defines  the  underlying   autoregressive  process   is  Harris
positive and we exhibit its stationary distribution.  As a consequence, we can rely on previously established
results for these autoregressive processes with Markov regime \citep{DMR_04} and thus obtain the
consistency  and asymptotic  normality  of the  MLE  for the  original
nearest-neighbor path in Markov environment. \\

Roughly  speaking, an  autoregressive model  with Markov  regime  is a
bivariate   process  where   the  first   component  forms   a  latent
(unobserved) Markov chain while conditionally on this first component,
the second one  has the structure of an  autoregressive process. These
processes  form a  generalization of  hidden Markov  models  (HMM), in
which 
the first  component remains a  latent Markov chain, while  the second
forms  a sequence  of independent  observations, conditionally  on the
first. HMM have been introduced by \cite{BP} with finite - latent and
observed - state spaces. Statistical properties of the MLE in HMM form a
rich literature;  a non exhaustive  list would start with  the seminal
work of \cite{BP}, include the developments of 
\cite{Leroux,BR,BRR,JP,Legland_Mevel,Douc_M,DMR_04,Genon_Catalot} 
and  finish  with the  latest  results  from  \cite{DMOvH}. A  general
introduction to HMM may be found in the  survey by \cite{Ephraim_Merhav} and the book
by \cite{Hmmbook}.

While it  is often believed that autoregressive  processes with Markov
regime are straightforward generalizations  of HMM (and this is indeed
the  case  concerning e.g.  algorithmic  procedures), the  statistical
properties of these  models are slightly more difficult  to obtain, \citep[see e.g.][for model selection issues]{Chambaz_M}. 
As
for  the  convergence properties  of  the  MLE,  only the  article  by
\cite{DMR_04} considers  the autoregressive case (instead  of HMM)
explaining why  we focus on their  results in our context.  It is also
worth  noticing that  many  of the  previous  results~\citep[with  exception
of][]{DMOvH} require uniform positivity of the transition density
of the latent Markov chain, which might not be satisfied in some applications
(particularly in the case of an unbounded  latent state space). As in
our case, the latent state  space corresponds to the environment state
space  and is  included  in $(0,1)$,  we  do not  face such  problems.
Moreover, we stress  that the results in \cite{DMR_04}  rely on rather
weak  assumptions \citep[compared  to previous  results in][on
which they are  partly built]{BRR,JP}. As a consequence,  the assumptions that
we obtain on  RWRE are also rather weak and will  be satisfied in many
contexts. \\

This   article   is   organized   as  follows.   Our   one-dimensional
nearest-neighbor   path  in   Markov  environment   is   described  in
Section~\ref{sec:RWMRE}.  Then we  explain why  the  direct likelihood
approach  may not  be  followed (Section~\ref{sec:pbm})  and cast  the
estimation  problem as  the one  of parameter  estimation in  a hidden
Markov model  (Section~\ref{sec:HMM}). After having set  the scene, we
state the assumptions (on the RWRE) and results in Section~\ref{sec:hyp_res}. We
prove  that (under classical  assumptions) the  MLE is  consistent and
asymptotically   normal.   Section~\ref{sec:exsim}   illustrates   our
results: we  start by  explaining how the  likelihood may  be computed
(Section~\ref{sec:comput}), then we  explore different examples and describe
our assumptions in  these cases (Section~\ref{sec:examples}) and close
the section with synthetic experiments on a simple example (Section~\ref{sec:simus}). The proofs
of  our results  are presented  in Section~\ref{sec:proofs}.  The main
point is to  establish that the bivariate Markov  chain that underlies
the HMM  is positive Harris  recurrent (Section~\ref{sec:propr}). Then
consistency, asymptotic normality  and efficiency (i.e. the asymptotic
variance is the inverse of the Fisher information) follow from~\cite{DMR_04} by
proving that our assumptions on the RWRE imply theirs on the HMM
(Sections~\ref{sec:cons_proof} and~\ref{sec:AN_proof}, respectively).

\section{Model description}
\subsection{Ballistic random walk in a Markov environment}
\label{sec:RWMRE}
We start this section by describing the random environment. 
Let $S$ be a closed subset of $(0,1)$ either finite, discrete or continuous, and $\B(S)$ the
associated Borel $\sigma$-field. 
The environment is given by $\bw=(\w_x)_{x\in \Z}\in S^\Z$, a positive Harris recurrent,
aperiodic and stationary
 Markov  chain with values in $S$ and transition kernel $Q:S\times\B(S)\to [0,1]$.
 We suppose that the transition kernel 
$Q=Q_\theta$  depends  on some  unknown  parameter  $\theta$ and  that
$\theta$ belongs to some compact space $\Theta \subset 
\R^q$. Moreover, $Q_\theta$ is absolutely continuous  
 either with respect to (w.r.t.) the Lebesgue measure on $(0,1)$ when $S$
is continuous  or w.r.t.  the counting measure  when $S$  is discrete,
with density denoted by $q_{\theta}$. 
 We  denote by $\mu_\theta$  the density  of its  associated stationary
distribution. 
Let us denote by $\PPt$ the law of the environment $\bw$ on $(S^\Z, \B(S^\Z))$ and $\EEt$ the corresponding expectation.
 
Now, conditionally on the environment, the law of the random walk $\bX=(X_t)_{t \in \N}$ is the one of the time homogeneous Markov chain on $\Z$ starting at $X_0=0$ and with  transition probabilities 
\begin{equation*}
\forall (x,y) \in \Z^2,\quad P_\bw(X_{t+1}=y|X_t=x) =
  \left\{
    \begin{array}{cc}
      \w_x & \text{ if } y=x+1, \\
1- \w_x & \text{ if } y=x-1, \\
0 & \text{ otherwise}.
    \end{array}
\right.
\end{equation*}
 The measure $P_\bw$ on $(\Z^{\N}, \B(\Z^\N))$
is usually referred to as the quenched law of walk $\bX$. 
Note  that this  conditional  law  does not  depend  on the  parameter
$\theta$ but only on the environment $\bw$ at the current site $x$.
 We also denote by 
$p_a(x,y)$ the corresponding transition density (w.r.t. to counting measure),
namely 
\[
\forall (x,y) \in \Z^2, \forall a \in S, \quad p_a(x,y) =a\1\{y=x+1\} + (1-a)\1\{y=x-1\},
\]
 where $\1\{\cdot\}$ denotes the indicator function.
Next we define the measure $\P^\theta$ on $S^\Z\times \Z^\N$ through
\begin{equation}
  \label{eq:Pt}
\forall F\in\B(S^\Z),\forall G\in\B(\Z^\N),\quad
\P^\theta(F\times G)=\int_F P_\bw (G) d\bf P^{\theta}(\w).
\end{equation}
The second marginal of $\P^{\theta}$ (that on $\Z^{\N}$), denoted also
$\P^\t$ when no  confusion occurs, is called the  annealed law of walk
$\bX$. 
We denote by $\E^\theta$ the corresponding expectation. Note
that the  first marginal  of $\P^\t$  is the law  of the  Markov chain
$\bw,$ denoted by $\PPt.$

For all $ k\in\Z$, we let 
\[   \tilde \w_k=\frac{1-\w_k}{\w_k}.\]
In the case of an i.i.d. environment  $\bw$, \cite{Sol} gives the classification of $\bX$ between transient or recurrent cases according to whether $\EEt (\log\tilde\w_0)$
  is different  or not from zero.   For stationary ergodic environments, which is the case here, \cite{Alili}   establishes   that  this
  characterization remains valid. 
Thus, if 
$
\EEt (\log \tilde\w_0 )<0,
$
then the walk is transient to the right, namely 
\[
\lim_{t\to\infty}X_t=+\infty, \quad \P^\t-a.s.
\]
Let $T_n$ be the first hitting time of the positive integer $n$, 
\begin{equation*} 
T_n = \inf \{ t \in \N \, : \, X_t = n \}
\end{equation*}
and define
\begin{equation}
  \label{eq:R}
R=(1+\tilde \w_1+\tilde\w_1\tilde\w_2 +\dots).
\end{equation}
Theorem 4.1 in \cite{Alili} shows that if
the environment satisfies the condition 
\begin{equation}
 \label{eq:bal}
\EEt (R) <+ \infty, 
\end{equation}
then the speed of the walk is strictly positive. Namely, $ \P^{\theta}
\mbox{-almost surely}$, the ratio  $T_n/n$ converges to a finite limit
as $n$ increases. Thus~\eqref{eq:bal}  gives the  so-called ballistic  condition  on the
random walk with Markov environment. Note that in the i.i.d. case, this condition reduces to $\EEt(\tilde \w_0)
  <1$. Moreover, in the non independent case, when the distribution of the environment is {\it uniquely
    ergodic}, namely $\bw$ is not i.i.d. and admits a unique invariant distribution, \cite{Alili} establishes
  that transience (namely $\EEt(\log \tilde \w_0) <0$) automatically implies ballistic regime \citep[see Lemma 6.1
  in][]{Alili}. Since in our context we assume that the Markov environment $\bw$ admits a unique invariant distribution,
  the ballistic assumption thus reduces to 
\begin{equation}
 \label{eq:ballistic} \left\{
   \begin{array}{ll}
     \EEt(\tilde \w_0) <1 & \text{ if } \bw \text{ i.i.d}, \\
\EEt(\log \tilde \w_0) <0 & \text{ if } \bw \text{ non independent}. 
   \end{array}
\right. 
\end{equation}
In the  following, we consider  a transient to the right  ballistic process $\bX$.

\subsection{Problem and motivation}
\label{sec:pbm}
We consider  a finite  trajectory $\bXn =(X_t)_{t  \le T_n}$  from the
process $\bX$, stopped at the first hitting time of a positive integer
$n\ge 1$. The apparently more general case of a sequence $(X_1,\dots,X_n)$ of observations is discussed in Remark~\ref{rem:nonseq}.
We  assume  that this  sequence  of  observations  is generated  under
$\P^{\ts}:=\Pstar$ for a true parameter value $\ts$ belonging to the interior $\mathring \Theta$ of $\Theta$. 
Our goal is  to estimate this parameter value  $\ts$ from the sequence
of observations $\bXn$ using a maximum likelihood approach.
To motivate the following developments, we will first explain why we can not directly rely on the likelihood of these observations.
Indeed, let $\mathcal{V}_{n}$ be the set of sites $x\in \Z$ visited by
the process up to time $T_n$, namely 
\[
\mathcal{V}_{n}=\{x\in \Z ; \,  \exists 0\le s \le T_n, X_s=x\}.
\]
Under the assumption of a transient (to the right) process, the random set $\mathcal{V}_{n}$ is equal to $\sem{\xi_n,n}$, where $\xi_n\in \Z^-$ and $\sem{a,b}$ denotes the set of integers between $a$ and $b$ for any  $a \le b$ in $\Z$. Here, $\xi_n$ is the smallest integer value visited by the process  $\bXn$. We also introduce $\bw(\bXn) :=(\w_{\xi_n},\dots,\w_n)$ which is the random environment restricted to the set of sites visited by $\bXn$. 
Now, the likelihood of $\bXn$ is given by the following expression
\begin{align}
 \Pt(\bXn)=                   &\int_{S}\dots                  \int_{S}
 \Pt(\bw(\bXn)=(a_{\xi_n},\dots,a_n),\bXn)da_{\xi_n}\dots da_n
\nonumber \\
 = &\int_S\dots \int_{S} \mu_\theta(a_{\xi_n}) \prod_{i=\xi_n}^{n-1} q_\theta(a_i,a_{i+1}) \prod_{s=0}^{T_n-1} p_{a_{X_s}}(X_{s},X_{s+1})
da_{\xi_n}\dots da_n.
\label{eq:naive_lik}
\end{align}

Computing the  likelihood from the  above expression would  require to
compute  $|\mathcal{V}_n|$  integral  terms (where  $|\cdot|$  denotes
cardinality).  As  $|\mathcal{V}_n|\ge n$,  this  means  that using  a
discretization method  over $N$ points  for each integral  (or letting
$N$ be the cardinality of $S$) would result in summing over at least $N^n$ different terms. This is unfeasible but for small values of $n$. Moreover, the above expression is not well suited for studying the convergence properties of this likelihood. 
Following  \cite{Comets_etal},  instead of  focusing  on the  observed
process $\bXn$, we will rather consider the underlying sequence $L_n^n
,  L_{n-1}^n ,  \dots ,  L_0^n$ of  the number  of left  steps  of the
process  $\bXn$  at  the   sequence  of  sites  $(n,n-1,\dots,0)$  and
construct our  estimator from this  latter sequence. Though we  do not
need it, note that it is argued in \cite{Comets_etal} that the latter is
in  fact a sufficient statistic  (at least  asymptotically)  for the
parameter $\theta$. In the next section,  we show that in the case
of  a  Markov  environment,  this  process exhibits  a  hidden  Markov
structure.   Moreover for  transient  RWRE,  this  process  is
  recurrent, allowing us to study the convergence properties of MLE.

\subsection{The underlying  hidden Markov chain}
\label{sec:HMM}
We define  the sequence of  left steps at  each visited site  from the
(positive part of the) trajectory $\bXn$ as follows. Let 
 \begin{equation*}
L_x^n :=\sum_{s=0}^{T_n-1}\1\{X_s=x;\ X_{s+1}=x-1\}, \quad \forall x \in \{0,\dots,n\}.
\end{equation*}
It is observed by \cite{KKS} in the case of an i.i.d. random environment 
 that  the  sequence  $(L_n^n  ,   L_{n-1}^n  ,  \dots  ,  L_0^n)$  is
 distributed  as a  branching  process with  immigration  in a  random
 environment (BPIRE). We will first show that this remains true in the
 case of a Markov environment.  To this aim, let us introduce the time reversed environment $\breve \bw=(\breve
   \w_x)_{x\in \Z}$ defined by $\breve \w_x =\w_{-x}$ for all $x\in \Z$. It is a Markov
   chain on $S$ with stationary density $\mu_\theta$ and transition $\breve q_\theta$ defined by 
\[
\forall a,b \in S, \quad \breve q_\theta (a,b) =\frac{\mu_\theta(b) q_\theta(b,a)}{\mu_\theta(a)}. 
\]
Now we recursively define a sequence  of random variables  $(Z_k)_{k\ge 0}$  with $Z_0=0$ and
\begin{equation}
  \label{eq:Z1}
\forall k \ge 0, \quad Z_{k+1} = \sum_{i=0}^{Z_k} \xi_{k+1,i},   
\end{equation}
where for all $k\ge  1$, the random variables $(\xi_{k,i})_{i\in\N }$,
are  defined  on  the   same  probability  space  as  previously,  are
independent and their  conditional distribution, given the environment
$\breve \bw$ is 
\begin{equation}
  \label{eq:Z2}
\forall m \in \N, \quad P_{\breve \bw}(\xi_{k,i}=m)= (1-\breve \w_k)^m\breve \w_k.   
\end{equation}
Here, $P_{\breve \bw}$ is defined similarly as $P_\bw$ for the environment $\breve \bw$ replacing $\bw$.
Then, conditionally on $\breve \bw$,  the sequence $(Z_k)_{k\in\N}$ is an
  inhomogeneous  branching process  with immigration,  with identical
offspring and immigration law, given by a geometric distribution (whose parameter depends on the random
environment $\breve \bw$).  Moreover,   it is easily seen that the annealed
distribution of the sequence $(L_n^n , L_{n-1}^n , \dots , L_0^n)$ and
that of $(Z_0 , Z_{1} , \dots , Z_n)$ are the same. 

\begin{lemm}\label{lem:eqinlaw}
For any fixed  integer $n\ge 1$, the sequence of left steps $(L_n^n , L_{n-1}^n , \dots , L_0^n)$ has same distribution as $(Z_0 , Z_{1} , \dots , Z_n)$ under $\Pt$.
\end{lemm}

\begin{proof}
For any fixed integer $n\ge 1$, let $\bar{\bw}^n:=(\w_n,\w_{n-1},\dots,\w_0,
  \w_{-1},\ldots)$ denote the time reversed environment starting at $\w_n$. 
Let also $P_{\bar{\bw}^n}$ be defined
  similarly as $P_\bw$ for the environment $\bar{\bw}^n$ replacing $\bw$.
Then  it is known  that  for any  sequence $(z_0,\dots,  z_n)\in
\N^{n+1}$, we have the equality 
\[
 P_\bw((L_n^n   , L_{n-1}^n , \dots , L_0^n)=(z_0,\dots, z_n))  =
P_{\bar\bw^n}((Z_0,\dots,Z_n)=(z_0,\dots,    z_n))  
\]
\citep[see for instance Section 4.1 in][]{Comets_etal}. 
Now the right-hand side $P_{\bar\bw^n}((Z_0,\dots,Z_n)=(z_0,\dots,    z_n))  $ only  depends on  the environment
  $\bar\bw^n$ 
through       the      first      $(n+1)$       variables $(\w_n,\w_{n-1},\dots,\w_0)$ whose distribution under $\PPt$
is the same as $(\breve \w_0,\breve \w_1,\dots,\breve\w_n)$.
As a consequence, using definition~\eqref{eq:Pt} of distribution $\Pt$ we obtain the result.  
\end{proof}

When the environment $\bw$ is composed of i.i.d. random variables, the
resulting  sequence $(Z_k)_{k\ge  0}$  is a  homogeneous Markov  chain
under $\Pt$ \citep[see e.g.][]{Comets_etal}. 
Now, when the  environment $\bw$ is itself a Markov chain, we  observe  that $(Z_k)_{k\ge 0}$ is distributed as a hidden Markov chain, or more precisely   
as the second marginal of a first order autoregressive process with Markov regime \citep{DMR_04}, where the latent
  sequence is given by $\breve \bw$. 
 We state this result as a lemma (namely Lemma~\ref{lem:bivariate} below)
 even though its proof is obvious and thus omitted. Let us recall that a first order autoregressive process with Markov
 regime (or Markov-switching autoregression) is a bivariate process $\{(\breve \w_k,Z_k)\}_{ k\ge 0}$ such that $\breve
 \bw=(\breve \w_k)_{k\ge 0}$ is a Markov chain and conditionally on $\breve \bw$, the sequence $(Z_k)_{k\ge
  0}$  is  an inhomogeneous  Markov chain whose transition from  $Z_{k-1}$ to $Z_k$
 only depends on $Z_{k-1}$ and $\breve \w_k$.  

For any $a\in S$ and $(x,y)\in \N^2$, denote
\begin{equation}
  \label{eq:emission}
g_a(x,y)=\binom{x+y}{x}a^{x+1}(1-a)^y 
\end{equation}
and let $\delta_x$ be the Dirac measure at $x$.
Let us recall that the process $(Z_k)_{ k\ge 0}$ is defined through~\eqref{eq:Z1} and~\eqref{eq:Z2}.
\begin{lemm}\label{lem:bivariate}
Under $\P^{\t}$, the  process $\{(\breve \w_k,Z_k)\}_{ k\ge  0}$ is a first-order
autoregressive  process   with  Markov  regime. The first component $\breve\bw$ is an homogenous  Markov chain with
transition kernel density $\breve q_\t$ and initial distribution $\mu_\t$.
 Conditionally   on  $\breve \bw$,  the   process  $(Z_k)_{k\in\N},$   is  an
 inhomogeneous Markov chain, starting from $Z_0=0$ and with transitions
\[\forall (x,y)\in \N^2, \forall k\in\N\quad
   P_{\breve\bw}(Z_{k+1}=y|Z_k=x)= g_{\breve \w_{k+1}}(x,y).
   \]
As a consequence, $\{(\breve \w_k,Z_k)\}_{ k\ge  0}$ is a Markov chain with state space $S\times \N$, starting from $\mu_\t\otimes \delta_0$ and 
with transition kernel density $\Pi_\t$ defined for all $(a,b,x,y) \in S^2\times\N^2$ by
\begin{equation}
\label{eq:bikernel}
 \Pi_\t((a,x), (b,y))  =  \breve q_\t(a,b)g_{b}(x,y). 
\end{equation} 
\end{lemm}

\begin{rema}
The  conditional  autoregressive   part  of  the  distribution,  given
by~\eqref{eq:emission}, is usually referred to as emission 
distribution. Note that in our  framework, this law does not depend on
the parameter $\theta$. 
\end{rema}

Under $\Pt$, the process $(\breve \bw,Z)$ has initial 
distribution $\mu_\t\otimes \delta_0$.  
In the sequel, we also need $(\breve \bw,Z)$ as well as the chain $(\w_k)_{k\ge 0}$  starting from any initial distribution.
For any probability $\nu$ on ${\cal B}(S\times \N)$,  denote  $\Pt_{\nu}$  the law 
  of $(\breve \bw,Z)$ starting  from $(\w_0,Z_0)\sim \nu$ (note that $\breve \w_0=\w_0$). Denote $\Et_{\nu}$
  the corresponding expectation.  In particular, for $(a,x)\in S\times
  \N$, we  let $\Pt_{(a,x)}$ and $\Et_{(a,x)}$ be  the probability and
  expectation  if  $( \w_0,Z_0)=(a,x)$.   Moreover, when  only  the
  first component  is concerned  and when no confusion occurs, if the chain  $\breve \bw$ or $\bw$ starts  from its
  stationary  distribution $\mu_{\t}$, we  still denote  this marginal
  law by $\PPt$ and the  corresponding expectation by $\EEt.$ If
  $\breve \bw$ or $\bw$
  start from another initial law,  for example $ \w_0=a$,  we denote
  their  law   by $\PPt_{a}$ and corresponding expectation  $\EEt_{a}$.
For  $n\in\N$,  we let $\F_n=\sigma\{\w_k, k=0,\ldots, n\}$ (resp. $\breve{\F}_n=\sigma\{\breve \w_k, k=0,\ldots, n\}$)  be
  the $\sigma$-field  induced by the $(n+1)$ first random variables    of the environment (resp. of the time reversed  environment). Moreover, we denote by 
$ \bw^n=( \w_n, \w_{n+1},\ldots)$ and 
$\breve \bw^n=(\breve \w_n,\breve \w_{n+1},\ldots)$ the shifted sequences. 
 The family of shift operators $(\tau^n)_{n\geq 1}$ where 
 $\tau^n:\Omega\to\Omega$ is defined by 
\begin{equation}
 \forall \bw, \breve \bw\in\Omega, \quad  \tau^n(\bw)= \bw^n \text{ and } \tau^n(\breve \bw)=\breve \bw^n.
 \end{equation} 
In  Section~\ref{sec:proofs},  we  show  that  under  the  ballistic
assumption, the  bivariate  kernel   $\Pi_\t((a,x),  (b,y))$  is  positive  Harris
recurrent and admits a unique invariant distribution with density $\pi_\t$, for
which we give an explicit formula (see Proposition~\ref{prop:biv_Harris}).  In  the following, we  let $\bPt$  and $\bEt$  be the  probability and
expectation induced  when considering the  chain $\{(\breve \w_k,Z_k)\}_{k\ge
  0}$  under its  stationary distribution  $\pi_\t$.

\section{Assumptions and results}
\label{sec:hyp_res}
\subsection{Estimator construction}

Recall that our aim is to infer the unknown parameter $\ts\in\mathring
\Theta$, using the observation  of a finite trajectory $\bXn$  up to the first
hitting time $T_n$ of site $n$. The observed trajectory is transformed
into the  sequence $L_n^n,
  L_{n-1}^n ,  \dots ,  L_0^n$ of  the number  of left  steps  of the
process  $\bXn$  at  the  sequence of  sites  $(n,n-1,\dots,0)$.  This
trajectory is  generated under the law  $\Pstar$ (recall that  $\Pstar$ is a
shorthand notation of $\Pts$). Due to the equality in law given by Lemma~\ref{lem:eqinlaw}, 
 we  can consider  that  we  observe a  sequence  of random  variables
 $(Z_0,\dots,Z_n)$ which is the  second component of an autoregressive
 process  with Markov  regime  described in  Lemma~\ref{lem:bivariate}.
 Thus under $\Pstar$, the law of the MLE of these observations is the same than the law of MLE built from $(Z_0,\dots,Z_n)$.
 
  As  a consequence,  we can  now rely  on a  set of  well established
  techniques developed in the context of autoregressive processes with
  Markov regime, both for computing efficiently the likelihood and for
  establishing its  asymptotic properties. Following~\cite{DMR_04}, we
  define a conditional log-likelihood, conditioned on an initial state
  of the environment  $\breve \w_0=a_0 \in S$. The reason for doing so is that the stationary distribution of
  $\{\breve \w_k,Z_k)\}_{k \ge 0}$ and hence the true likelihood, is typically infeasible to compute. 

\begin{defi}\label{def:crit}
Fix some  $a_0\in S$ and  consider the conditional  log-likelihood of
the observations defined as 
\begin{equation}
  \label{eq:log_lik}
\ell_n(\t,a_0):=\log\P^{\t}_{(a_0,0)}(Z_1,\dots,
Z_n)=\log\int_{S^n}\prod_{i=1}^n
\breve q_{\t}(a_{i-1},a_i)g_{a_i}(Z_{i-1},Z_i)da_i .
\end{equation}
\end{defi}

Note  that the  above expression  of the  (conditional) log-likelihood
shares the computational problems mentioned for expression~\eqref{eq:naive_lik}.  However, in the present context of
autoregressive processes with  Markov regime, efficient computation of
this expression  is possible. The key ingredient  for this computation
(that also serves to study the convergence properties of $\ell_n$)
is to rely on the following additive form 
\begin{align*}
  \ell_n(\theta,a_0)   & =    \sum_{k=1}^n    \log    \Pt    _{(a_0,0)}
  (Z_k|Z_0,\dots,Z_{k-1}) \\
&= \sum_{k=1}^n \log \left( \iint_{ S^2}g_{b}(Z_{k-1}, Z_k)
\breve q_\t(a,b)\Pt_{(a_0,0)}(\breve \w_{k-1}=a|Z_0^{k-1}) da db
\right) ,
\end{align*}
where $Z_s^t$  denotes $Z_s,Z_{s+1},\dots,Z_t$ for  any integers $s\le
t$.  
 We further develop this point in Section~\ref{sec:comput} and also refer to \cite{DMR_04} for more details.

\begin{defi}
  The estimator $\hat \theta_n$ is defined as a measurable choice\[
\hat \theta_n \in \argmax_{\theta \in \Theta} \ell_n(\t,a_0).
\] 
\end{defi}
Note  that  we omit  the  dependence  of  $\hat \theta_n$  on  the
  initial state $a_0$ of the environment. 

  \begin{rema}\label{rem:nonseq}
When considering a size-$n$ sample $X_1, \dots X_n$ instead of a trajectory stopped at random time $T_n$, we may
consider $m:= m(n) = \max_{1\le i \le n } X_i$ and restrict our attention to the sub-sample $X_1,\dots, X_{T_m}$. As we consider a transient random walk, $m(n)$ increases to infinity with $n$. Consistency with respect
  to $n$ or $m(n)$ is equivalent. Now considering the rates, note that in the ballistic case we can easily obtain
    that $m(n) \sim cn$ for some $c>0$ so that rates of convergence as well as efficiency issues with
  respect to $m(n)$ or $n$ are the same. Note that  information about $T_{m}$ has to be extracted first and
  then the data may be reduced to the sequence of left steps without loosing information. 
  \end{rema}

\subsection{Assumptions and results}
Recall  that  $ q_{\theta}$  and  $\mu_{\theta}$ are  respectively  the
transition  and the  invariant probability densities  of the  environment Markov
chain $ \bw$ with values in $S$, while  $\breve q_{\theta}$  and  $\mu_{\theta}$ are the same quantities for the time
reversed chain $\breve \bw$. Moreover, $S$ is a closed
subset  of  $(0,1)$ so  that  we can  assume  that  there exists  some
$\eps \in (0,1)$ such that 
\begin{equation}
  \label{eq:S_subset}
S \subseteq [\eps; 1-\eps].  
\end{equation}
The above assumption is known as the uniform ellipticity condition. 

We also recall that the random variable $R$ is defined by~\eqref{eq:R}.
 
\begin{assu}(Ballistic case). 
  \label{hyp:ballistic}
For any $\theta\in\Theta$, Inequality~\eqref{eq:ballistic} is satisfied. 
\end{assu}

\begin{assu} 
\label{hyp:noyau}
There exist some constants $0<\sigma_{-}, \sigma_+<+\infty $ such that 
\begin{equation*}
\sigma_{-}\le\inf_{\t \in \Theta} \inf_{a,b \in S} q_\t(a,b)\le \sup_{\t \in \Theta} \sup_{a,b\in S}q_{\t}(a,b)\le \sigma_+.
  \end{equation*}
\end{assu}
Note that it  easily follows from this assumption  that the stationary
density $\mu_\t$ also satisfies 
    \begin{equation}\label{eq:stat_bound}
\sigma_{-}\le\inf_{\t \in \Theta} \inf_{a \in S} \mu_\t(a)\le  \sup_{\t \in \Theta} \sup_{a\in S}\mu_{\t}(a)\le \sigma_+.
  \end{equation}
Moreover, we also get that the time reversed transition density $\breve q_\theta$ satisfies
  \begin{equation}
    \label{eq:breve_bound}
   \frac{ \sigma_{-}^2} {\sigma_{+}} \le\inf_{\t \in \Theta} \inf_{a,b \in S} \breve q_\t(a,b)\le \sup_{\t \in \Theta} \sup_{a,b\in S}\breve q_{\t}(a,b)\le \frac{\sigma_+^2}{\sigma_{-}}.
  \end{equation}

Assumptions~\ref{hyp:ballistic} and~\ref{hyp:noyau} are used to establish that the bivariate process $\{ (\breve \w_k,
  Z_k)\}_{k\geq 0}$ is positive Harris recurrent. Note in particular that the weakest assumptions currently ensuring consistency
  of the MLE in the HMM setting contain positive Harris recurrence of the hidden chain \citep{DMOvH} and~\ref{hyp:noyau} is
  further required in the less simple case of an autoregressive model with Markov regime \citep{DMR_04}. The lower bound 
  in~\ref{hyp:noyau} may be  restrictive in a general  HMM setting as it prevents the support $S$ from being unbounded. However here
  we have $S\subseteq  (0,1)$ and thus~\ref{hyp:noyau} is satisfied in many examples (see Section~\ref{sec:exsim}). 

Next assumption is classical from a statistical perspective and requires smoothness of the underlying model.  

  \begin{assu}(Regularity condition). 
  \label{hyp:regular}
For all $(a,b)\in S^2$, the map $\t \mapsto q_\t(a,b)$ is continuous.
\end{assu}
In order to ensure identifiability of the model, we naturally require identifiability of the parameter from the distribution of the environment. 

\begin{assu}(Identifiability condition).
  \label{hyp:ident}
  \begin{equation*}
\forall \t,\t' \in \Theta, \quad 
    \theta = \theta' \iff q_\t = q_{\theta'}.
  \end{equation*}
\end{assu}

\begin{theo}
\label{thm:consistency}
  Under  Assumptions~\ref{hyp:ballistic}  to~\ref{hyp:ident}, the  maximum
  likelihood estimator $\hat \theta_n$ converges $\Pstar$-almost surely to the true parameter
  value $\ts$ as $n$ tends to infinity. 
\end{theo}
We now introduce the  conditions that will ensure asymptotic normality
of $\hat\theta_n$ under $\Pstar$.  In the following, for any function $\varphi :\Theta
\mapsto \R$, we let $\partial_\theta \varphi$ and
$\partial^2_\theta  \varphi$  denote  gradient  vector and  Hessian  matrix,
respectively. Moreover, $\|\cdot \|$ is  the uniform norm (of a vector
or a matrix). Again, next condition is classical and requires regularity of the mapping
  underlying the statistical model.

\begin{assu}
  \label{hyp:regular2}
For  all $(a,b)\in  S^2$,  the  map $\t  \mapsto  q_\t(a,b)$ is  twice
continuously differentiable on $\mathring\Theta$. Moreover, 
\begin{align*}
  \sup_{\t  \in   \mathring  \Theta}\sup_{a,b  \in   S}  \|\partial_\t
  \log q_\t(a,b)\| <+\infty , & \quad   \sup_{\t  \in   \mathring  \Theta}\sup_{a,b  \in   S}  \|\partial^2_\t  \log
  q_\t(a,b)\| <+\infty , \\
 \sup_{\t  \in   \mathring  \Theta}\sup_{a  \in   S}  \|\partial_\t
 \log \mu_\t(a)\| <+\infty  & \text{ and }   \sup_{\t  \in   \mathring  \Theta}\sup_{a  \in   S}  \|\partial^2_\t  \log \mu_\t(a)\| <+\infty .
\end{align*}
\end{assu}

Following the notation from Section 6.1 in~\cite{DMR_04}, we now introduce the
asymptotic Fisher  information matrix. 
We start by extending the chain $\{(\breve \w_k,Z_k)\}_{k\in\N}$ with indexes
in $\N$ to a stationary Markov chain $\{(\breve \w_k,Z_k)\}_{k\in\Z}$ indexed
by $\Z$. 
Let  us recall that  $\bPt$ and
$\bEt$  respectively  denote  probability  and expectation  under  the
stationary distribution  $\pi_\t$ of the  chain $\{(\breve \w_k,Z_k)\}_{k \ge
  0}$.
For any $k\ge 1, m\ge 0$, we let 
\begin{align*}
  \Delta_{k,m} (\t)
= &\bEt \Big(\sum_{i=-m+1}^k \partial_\t \log \breve q_\t(\breve \w_{i-1},\breve \w_i) \big|
Z_{-m}^k \Big) \\
&-\bEt  \Big(\sum_{i=-m+1}^{k-1}  \partial_\t  \log  \breve q_\t(\breve \w_{i-1}, \breve \w_i)
\big| Z_{-m}^{k-1} \Big). 
\end{align*}
 Note that this expression derives from Fisher identity stated in
\cite{Louis}.  Indeed, under general  assumptions, the  score function
equals the  conditional expectation of  the complete score,  given the
observed data. As the emission distribution $g$ does not depend on the
parameter  $\theta$, the  complete score  reduces  to a  sum of  terms
involving $\breve q_\t$ only. 

Lemma 10 in~\cite{DMR_04} establishes that for any $k\ge 1$, the
sequence      $(\Delta_{k,m}(\ts))_{m\ge     0}$      converges     in
$\mathbb{L}^2(\bPs)$ 
   to     some     limit    $\Delta_{k,\infty}(\ts)$.
 From this quantity, we may define 
\begin{equation}
  \label{eq:Fisher}
  I(\ts) = \bEs(\Delta_{0,\infty} (\ts) ^\intercal \Delta_{0,\infty} (\ts) ),
\end{equation}
where by convention $\Delta_{0,\infty} $ is a row vector and 
$u^\intercal$ is the transpose vector of $u$. Then, $ I(\ts)$ is
the  Fisher information  matrix of  the model.  We can  now  state the
asymptotic normality result.

\begin{theo}\label{thm:AN}
  Under  Assumptions~\ref{hyp:ballistic}   to~\ref{hyp:regular2},  if  the
  asymptotic    Fisher    information    matrix    $I(\ts)$    defined
  by~\eqref{eq:Fisher} is invertible, we have that 
\[
n^{-1/2}(\hat\theta_n-\ts) \mathop{\longrightarrow}_{n \to +\infty} \mathcal{N}(0,I(\ts)^{-1}), \quad \Pstar\text{-weakly}.
\]
\end{theo}

Note  that the  definition of  $ I(\ts)  $ is  not  constructive. In
particular, asymptotic normality of the MLE requires that $I(\ts)$ is invertible but this
may not  be ensured through  more explicit conditions on  the original
process. However,  the Fisher information may  be approximated through
the Hessian of  the log-likelihood. Indeed, Theorem~3 in~\cite{DMR_04}
states that the normalized  Hessian of the log-likelihood converges to
$-I(\ts)$ under stationary  distribution $\bPs$. Moreover, this result
is generalized to obtain convergence under non stationary distribution
$\Pstar$ (see the proof of Theorem 6 in that reference). Thus we have 
\begin{equation}
  \label{eq:estim_Fisher}
\frac 1 n \partial^2_\t \ell_n(\hat \t_n) \mathop{\longrightarrow}_{n \to +\infty} - I(\ts), \quad \Pstar-\text{a.s.}
\end{equation}
In practice, this  may be used to approximate  the asymptotic variance
of the estimator $\hat \t_n$, as illustrated in Section~\ref{sec:simus}.

\section{Illustration: examples and simulations} \label{sec:exsim}

\subsection{Computation of the likelihood}
\label{sec:comput}

The  computation of  the log-likelihood  relies  on the  following set  of
equations. As already noted, we have 
\begin{align}
  \ell_n(\theta,a)  =&\sum_{k=1}^n
  \log \Pt_{(a,0)}(Z_k|Z_0^{k-1}), \nonumber \\
=& \sum_{k=1}^n \log \left( \iint_{S^2}g_{b'}(Z_{k-1}, Z_k)
  \breve q_\t(b,b')\Pt_{(a,0)}(\breve \w_{k-1}=b|Z_0^{k-1})  db db'
\right) .\
\label{eq:lik_filter}
\end{align}
In this expression, the quantity 
\begin{equation}
  \label{eq:filter}
F^{\t,a}_k(\cdot) = \Pt_{(a,0)}(\breve \w_k=\cdot|Z_0^{k}) ,
\end{equation}
is the called the {\it prediction filter}. It is a probability distribution on
$S$ and it is computed through recurrence relations. Indeed, we have 
\begin{equation}
  \label{eq:filter_rec}
\left\{
  \begin{array}{l}
    F_0^{\t,a} = \delta_a, \\
F_{k+1}^{\t,a}    (b')   \propto    g_{b'}(Z_k,Z_{k+1})\int_{S}
\breve q_\t(b,b') F_{k}^{\t,a} (b) db , \quad k\ge 0, b'\in S, 
  \end{array}
\right. 
\end{equation}
where 
$\propto$ means
proportional to (up to a normalizing constant).

When $S$ is  discrete, the integral terms over $S$  reduce to sums and
computing  the prediction  filter recursively  enables to  compute the
log-likelihood of the observations, and  then the MLE. We illustrate these
computations  in the  case of  Example~\ref{ex:Markov_fini}  below as
well as in Section~\ref{sec:simus}.
When   $S$  is   continuous,  approximation   methods   are  required,
e.g.  particle filters  or Monte  Carlo expectation-maximisation (\texttt{em})
algorithms. We refer to Section 8 in~\cite{DMR_04} for more details.

Note that  in any case,  optimisation of the log-likelihood  is either
done through \texttt{em} 
algorithm~\citep{BPSW,DempsterLR} or by direct  optimisation procedures, as there is
no analytical expression for its 
maximiser. Thus, the computation of the gradient of this log-likelihood is
often     used    (e.g.     in    descent     gradient    optimisation
methods).  
As soon  as we can differentiate  under the integral  sign (which is valid under Assumption~\ref{hyp:regular2}), the gradient function
$\partial_\t \ell_n(\t, a)$ writes 
\begin{align}
& \partial_\t \ell_n(\t, a) = 
\Big( \iint_{S^2}g_{b'}(Z_{k-1}, Z_k)
 \breve q_\t(b,b')F_{k-1}^{\t, a} (b) db db' \Big) ^{-1} \nonumber \\
&\times \sum_{k=1}^n \Big( \iint_{S^2}g_{b'}(Z_{k-1}, Z_k)
 [ \partial_\t \breve q_\t(b,b') F_{k-1}^{\t, a} (b) + \breve q_\t(b,b') \partial_\t F_{k-1}^{\t, a} (b) ]
db db' \Big) .
\label{eq:gradient}
\end{align}
Note  that   the  gradient  of  the   prediction  filter  $\partial_\t
F_{k-1}^{\t, a} $ may be obtained through recurrence relations similar
to~\eqref{eq:filter_rec}. However,  these relations are  more involved
since the normalizing constant in~\eqref{eq:filter_rec} depends on $\t$ and
can not be neglected. 

To conclude this section, we mention that computing the Hessian of the
log-likelihood can be done in a  similar way.

\subsection{Examples}
\label{sec:examples}

In this  section, we provide  some examples of environments  $\bw$ and
check the assumptions needed  for consistency and asymptotic normality
of the MLE.

\begin{exam}
(Simple i.i.d. environment on two values.) Let $q_\t (a,b)= \mu_\t(b)$ and $\mu_\t(\cdot)=\mu_p(\cdot) =p\delta_{a_1} (\cdot) +(1-p)\delta_{a_2} (\cdot) $
with  known  values $a_1,a_2\in  (0,1)$  and  unknown parameter  $p\in
[\gamma,1-\gamma]\subseteq (0,1)$. 
\end{exam}

The  support  of  the
environment is  reduced to $S=\{a_1,a_2\}$.  Moreover,  we assume that
$a_1,a_2$ and $\Theta$ are such that the process is transient
to the right and ballistic.  
In the i.i.d. case, the
ballistic assumption (that also implies transience) reduces to $\EEt(\tilde \w_0) <1$ and thus to 
\[
p \frac{1-a_1}{a_1} + (1-p)\frac{1-a_2}{a_2} <1.
\]
The  log-likelihood of the  observations has  a very  simple form  in this
setup 
\[
\ell_n(p)  =  \sum_{k=1}^n  \log  \big[p a_1^{Z_{k-1}+1}(1-a_1)^{Z_k}  +  (1-p)
a_2^{Z_{k-1}+1} (1-a_2)^{Z_k} \big],
\]
and its maximiser $\hat \t_n=\hat p_n$ is obtained through numerical
optimisation. 
We refer to~\cite{Comets_etal,Falc_etal} for previous results obtained
in this setup. 

Assumptions~\ref{hyp:noyau} and \ref{hyp:ident} are satisfied as soon as $\Theta \subseteq
[\gamma,1-\gamma]$    and   $a_1\neq    a_2$,    respectively.   Moreover,
Assumptions~\ref{hyp:regular} and \ref{hyp:regular2} are automatically
satisfied. Indeed, for any $p\in \Theta$ and any $a\in S$, we have 
\begin{align*}
& | \partial_p \log \mu_p(a) |= \frac{1}{p\1\{a=a_1\}+(1-p)\1\{a=a_2\}} , \\ 
& | \partial_p^2 \log \mu_p(a) |= \frac{1}{p^2\1\{a=a_1\}+(1-p)^2\1\{a=a_2\}}.
\end{align*}
As a consequence,  Theorems~\ref{thm:consistency} and \ref{thm:AN} are
valid in this setup.

\begin{exam}\label{ex:Markov_fini}
  (Finite Markov chain environment.)
Let us assume that $S=\{a_1,a_2\}$ is fixed and known and the stationary
Markov chain $\bw$ is defined through its transition matrix 
 \begin{equation*}
Q_ \theta = 
  \begin{pmatrix}
    \alpha & 1-\alpha\\
1-\beta & \beta
  \end{pmatrix} ,
 \end{equation*}
where     the      parameter     is     $\theta=(\alpha,\beta)     \in
[\gamma,1-\gamma]^2$ for some $\gamma >0$ 
\end{exam}

Note that Assumption~\ref{hyp:noyau} is satisfied as soon as $\gamma>0$.  
The stationary measure of the Markov chain is given by 
\[
\mu_\t=\Big(           \frac         {1-\beta}           {2-\alpha-\beta},\frac
{1-\alpha}{2-\alpha-\beta} \Big ) .
\]
This is automatically a reversible Markov chain so that $\breve q_\theta = q_\theta$.
The transience condition writes
\[
(1-\beta)\log\left(\frac{1-a_1}{a_1}\right) + (1-\alpha)\log\left(\frac{1-a_2}{a_2}\right) <0 .
\]
Moreover, as soon as $\alpha \neq 1- \beta$ the sequence $\bw$ is non independent and the existence of a unique
  stationary measure for $\bw$ ensures the ballistic regime from transience assumption~\citep[Lemma 6.1 in][]{Alili}.
Let  us now  consider  the  log-likelihood expression  in  this setup.  As
already explained,  the key point  for computing the log-likelihood  in the
setup of  an autoregressive process with  Markov regime is  to rely on
the following additive form 
\begin{align*}
  \ell_n(\theta,a_1) = &\sum_{k=1}^n \log \Pt(Z_k|Z_0^{k-1},\w_0=a_1)\\
=& \sum_{k=1}^n \log \left( \sum_{b,b'\in S^2}g_{b'}(Z_{k-1}, Z_k)
  q_\t(b,b') F_{k-1}^{\t,a_1} (b) 
\right) ,
\end{align*}
where $F^{\t,a}_k$ is the prediction filter defined by~\eqref{eq:filter} and we used $\breve q_\theta=q_\theta$.
Relying  on matrix  notation, we  let $F_k^{\t,a}$  be the  row vector
$(F_k^{\t,a}(a_1), F_k^{\t,a}(a_2))$ while $G_{k}$ is the row vector 
$(g_{a_1}(Z_{k-1},Z_{k}),  g_{a_2}(Z_{k-1},Z_{k}))$  and $u^\intercal$
the transpose vector of $u$. Then we obtain 
\begin{equation*}
    \ell_n(\theta,a_1)   =   \sum_{k=1}^n   \log  \big[F_{k-1}^{\t,a_1}   Q_\t
  G_{k}^\intercal\big].
\end{equation*}
Moreover, the sequence of  prediction filters $\{F_k^{\t, a_1}\}_{0\le k
  \le n-1}$ is obtained through the
recurrence relations~\eqref{eq:filter_rec} that in our context, write as
\begin{equation*}
  \left\{
    \begin{array}{l}
      F_0^{\t,a_1} =(1,0)\\
F_{k+1}^{\t, a_1} \propto F_{k}^{\t,a} Q_\t \text{Diag}(G_{k+1}) .
    \end{array}
\right.
\end{equation*}
Now,   the   gradient   function  $\partial_\t   \ell_n(\t,a)$   given
by~\eqref{eq:gradient} satisfies the
following equations
\begin{equation*}
\left\{
  \begin{array}{l}
     \partial_\alpha \ell_n(\t,a) = \sum_{k=1}^n \Big[(\partial_\alpha F_{k-1}^{\t,a_1}
    Q_\t + F_{k-1}^{\t,a_1} Q'_1) 
  G_{k}^\intercal \Big] \Big(F_{k-1}^{\t,a_1} Q_\t
  G_{k}^\intercal \Big)^{-1}, \\
    \partial_\beta \ell_n(\t,a) = \sum_{k=1}^n \Big[(\partial_\beta F_{k-1}^{\t,a_1}
    Q_\t + F_{k-1}^{\t,a_1} Q'_2) 
  G_{k}^\intercal \Big] \Big(F_{k-1}^{\t,a_1} Q_\t
  G_{k}^\intercal \Big)^{-1}, 
  \end{array}
\right.
\end{equation*}
where $\partial_{i}  F_{k-1}^{\t,a_1}$ is  the row vector  with entries
$(\partial_{i}            F_{k-1}^{\t,a_1}(a_1),            \partial_{i}
F_{k-1}^{\t,a_1}(a_2))$ and 
\[
Q'_1 =   \begin{pmatrix}
    1 & -1\\
0 & 0
  \end{pmatrix} ,
\quad 
Q'_2=   \begin{pmatrix}
    0 & 0\\
-1 & 1
  \end{pmatrix} .
\]
Let us denote by $\boldsymbol{1}$ the row vector $(1,1)$. 
In  the current  setup, the  derivative  of the  prediction filter  is
obtained through $\partial_\alpha F_0^{\t,a_1}=\partial_\beta F_0^{\t,a_1}
=(0,0)$ and for any $k\ge 0$, 
\begin{align*}
  \partial_\alpha F_{k+1}^{\t, a_1} =&
\Big(  F_k^{\t,    a_1} Q_\t \text{Diag}(G_{k+1}) 
\boldsymbol 1 ^{\intercal} \Big)^{-1} \times 
\Big( \partial_\alpha F_k^{\t, a_1} Q_\t + F_k^{\t, a_1} Q'_1\Big) \text{Diag}(G_{k+1}) \\
&-  \frac {\Big[\Big(  \partial_\alpha F_k^{\t,  a_1} Q_\t  + F_k^{\t,  a_1} Q'_1\Big)
\text{Diag}(G_{k+1}) \boldsymbol 1 ^{\intercal} \Big]} 
{\Big(  F_k^{\t,    a_1} Q_\t \text{Diag}(G_{k+1}) 
\boldsymbol 1 ^{\intercal}\Big)^{2}}
\times F_k^{\t, a_1} Q_\t \text{Diag}(G_{k+1}) ,
\end{align*}
and a similar equation holds for $ \partial_\beta F_{k+1}^{\t, a_1} $.

In  Section~\ref{sec:simus},  we  provide  an
  illustration of the numerical performances of the maximum likelihood estimator
  in this setup. Note that  second order derivatives of the prediction
  filter and thus the log-likelihood are obtained similarly. These are
  used to estimate the asymptotic covariance matrix of the MLE in Section~\ref{sec:simus}.

To conclude this section, note that the regularity assumptions~\ref{hyp:regular}
and~\ref{hyp:regular2} are  satisfied, as well  as the identifiability
condition~\ref{hyp:ident},  as soon as  $a_1\neq a_2$  and $\alpha\neq
\beta$. As a consequence, Theorems~\ref{thm:consistency} and \ref{thm:AN} are
valid in this setup.

\begin{exam}(DNA unzipping.)
We consider the context of DNA unzipping studied  in \cite{Balda_06} where
the goal is the sequencing of a molecule \citep[see also][]{Andreo}. The physical experiment
consists in observing many different  unzippings of a DNA molecule which, due to
its physical properties, may naturally (re)-zip.   
 In this context, the random walk $\bX$ represents the position of the
 fork at each  time $t$ of the experiment,  or equivalently the number
 of currently unzipped bases of the molecule. In the previous works, the authors are interested in the  observation of many
  finite trajectories of the random walk in this random environment. 
Here, we consider the different problem of a single unzipping of a
   sufficiently long molecule.
\end{exam}

Let $\A=\{A,C,G,T\}$ denote the finite nucleotide alphabet.  
The sequence of bases
$\{b_x\}_{1\le x \le n } \in \A^n $ of the (finite length) molecule are unknown and
induce a specific environment that  will be considered as random. More
precisely, the conditional transitions of the random walk are given by 
  \begin{equation*} 
 \w_x= \frac{1}{1+\exp(\beta g({x},{x+1}))}.
 \end{equation*}
where $g({x},{x+1}):=g_0(b_{x},b_{x+1})-g_1(f)$.   The known parameter
$g_1(f)$ is the work to stretch under a force $f$ the open part of the
two  strands,  it   can  be  adjusted  but  is   constant  during  the
unzipping. Parameter  $\beta>0$ is also known and  proportional to the
inverse of temperature. The quantity $g_0(b_{x},b_{x+1})$ is the binding energy that takes into account additional stacking effects and therefore depends on the base values 
at the $(x+1)$-th and also at the $x$-th positions.  
Table~\ref{tab1} gives these binding energies at room temperature \citep[see][]{Balda_06}.
\begin{table}[h]
 \begin{center}
\begin{tabular}{|l|l|l|l|l|}
\hline   $g_0$ & A & T & C & G   \\
\hline   A  &  1.78 & 1.55  & 2.52 & 2.22  \\
\hline
 T  &  1.06 & 1.78  & 2.28 & 2.54  \\
\hline
 C  &  2.54 & 2.22  & 3.14 & 3.85  \\
\hline
 G  &  2.28 & 2.52  & 3.90 & 3.14  \\
\hline
\end{tabular} 
\end{center}
\caption{Binding free energies (units of $k_BT$).} 
\label{tab1}
\end{table}
To take into account this dependence between energies, we assume that
$\{g_0(x):=g_0(b_{x},b_{x+1})\}_{x \geq 1}$ is a Markov chain. With
this assumption and since the mapping $g_0(x) \mapsto \w_x$ is one-to-one, $\bw=(\omega_x)_{x \geq 1}$ is Markov as well.  The
parameter of  the model is  thus the transition matrix  $Q_\t$ between
the binding energies.  Note that while the set of dinucleotides has
cardinality $16$, function $g_0$ takes only 10 different values.
So random environment $\bw$ takes values in $S$ with cardinality 10 
and the underlying transition matrix $Q_\t$ (for the binding energies)
is of size $10\times 10$ but has many
zero  entries

 The ballistic condition is not difficult to satisfy. Indeed, we have 
 \[
 \tilde \omega_x= \exp(\beta (g_0(x)-g_1(f)))
 \]
 and $g_1$ is increasing with $f$. Thus we may  choose $f$ such that $g_1$
 is large enough to ensure  either $\EEt(\tilde \omega_0)<1$ 
if the sequence $\{g_0(x)\}_{x\ge 1}$ is only i.i.d. or to ensure $\EEt(\log \tilde \omega_0)<0$  when
 the sequence $\{g_0(x)\}_{x\ge 1}$ is not independent. In both cases, this ensures the ballistic regime.

In this context and for a long enough sequence, we can estimate the matrix $Q_\t$ of the transitions
between  the different  binding  energies, as  well as  $\mu_{\theta}$
which gives the frequencies of  appearance of the binding energies. In
turn, this also  gives the frequencies of appearance  of certain base
pairs thanks to Table~\ref{tab1}.
Since  both  parameter space  $\Theta$  and  state  space $S$  are
  finite,                Assumptions~\ref{hyp:noyau},~\ref{hyp:regular}
  and~\ref{hyp:regular2}  are satisfied.   This is  also the  case for
  identifiability assumption~\ref{hyp:ident}. As a consequence, Theorems~\ref{thm:consistency} and \ref{thm:AN} are
valid in this setup.

\begin{exam} (Auto-regressive environment.) Let $y_0\sim \mu_\t$ and
  for any $n\ge 0$, we let $y_{n+1}=\alpha y_n+u_n$ 
  where $\alpha \in \R$ and $(u_n)_{n\ge 0}$
  is an i.i.d. sequence. Fix some $\eps >0$. The environment $\bw$
  is defined on $S=[\eps,1-\eps]$ through a truncated logistic function 
\[
\w_n = \phi_\eps(y_n):= \left\{
  \begin{array}[h]{ll}
  e^{y_n} (1+e^{y_n})^{-1}& \text{if } e^{y_n} (1+e^{y_n})^{-1}\in S ,\\
\eps & \text{if }  e^{y_n} (1+e^{y_n})^{-1} \le \eps,\\
1-\eps & \text{if }  e^{y_n} (1+e^{y_n})^{-1}\ge 1-\eps.
  \end{array}
\right.
\]
\end{exam}

A time reversibility condition on first-order autoregressive processes is
studied in \cite{Osawa}. 
If  we assume  that $u_n$  has Gaussian  distribution,  say $u_n
\sim\mathcal{N}(\mu,\sigma^2)$, then for any value $|\alpha|<1$, it is
easily seen that 
there exists a stationary density $\mu_\t$ for $(y_n)_{n \ge 0}$ given by 
\[
\forall y \in \R, \quad \mu_\t (y)= \left( \frac {1-\alpha^2} {2\pi
    \sigma^2} \right)^{1/2}
\exp\left[ -\frac {1-\alpha^2} {2\sigma^2}
\Big\{ y- \frac{\mu } {1-\alpha}\Big\}^2 \right], 
\]
where $\theta=(\alpha,\mu,\sigma^2)$ is the model parameter. Moreover,
the process $(y_n)_{n \ge 0}$ is reversible w.r.t. this stationary distribution. 
Then   $(\w_n)_{n\ge   0}$   is   also   stationary   and time  reversible.  

Note   that   the   inverse   function   $\phi_\eps^{-1}   :   S   \to
[\log(\eps/(1-\eps)), \log((1-\eps)/\eps)]$ is well defined and
given by 
\[
\forall  a   \in  S,  \quad  \phi_\eps^{-1}(a)   =\log  \left(\frac  a
  {1-a}\right) .
\]
The   transience  condition   writes  $\EEt(y_0)>0$   or  equivalently
$\mu>0$.  As soon as $\alpha \neq 0$, the sequences $(y_n)_{n \ge 0}$ and thus also $(\w_n)_{n\ge 0}$ are non
  independent and the existence of a unique stationary distribution implies the ballistic regime from transience
  assumption \citep[Lemma 6.1 in][]{Alili}.
Now, the transition density of $\bw$ is given by 
\[
q_\t(a,b) = \frac 1 {\sqrt{2\pi}\sigma b(1-b)}
\exp\Big( - \frac 1 {2 \sigma^2}
 (\phi_\eps^{-1} (b) -\alpha \phi_\eps^{-1}(a) -\mu)^2 \Big).
\]
As a  consequence, Assumption~\ref{hyp:noyau} is satisfied  as soon as
$\sigma^2$ and  $\mu$ are bounded.  Thus we assume that  the parameter
space satisfies 
\[
\Theta     =     \A\times    [\mu_{\text{min}},\mu_{\text{max}}]\times
[\sigma_{\text{min}},\sigma_{\text{max}}] , 
\]
where $\A$ is a compact subset of $(-1,1)$ and the constants satisfy $
\mu_{\text{min} } >C(\eps)+\sigma^2_{\text{max}}/2$ and 
$\sigma_{\text{min}} >0$.  Moreover, regularity assumptions~\ref{hyp:regular}
and \ref{hyp:regular2} are also  satisfied, as well as identifiability
condition~\ref{hyp:ident}.
As a consequence, Theorems~\ref{thm:consistency} and \ref{thm:AN} are
valid in this setup.


\subsection{Numerical performance}
\label{sec:simus}

In  this section,  we illustrate  our results  in the  simple  case of
Example~\ref{ex:Markov_fini}. We start by describing the experiment. 
The support of the environment is fixed to $S=\{0.4,0.8\}$. The true
parameter value is chosen as $(\alpha^\star,\beta^\star)=(0.2,0.9)$.  These
choices  ensure transience  of the  walk as  well as  ballistic regime. Next,  we repeat  100  times the
following procedure. 
We    simulate    a    RWRE    under   the    model    described    in
Example~\ref{ex:Markov_fini}        with        parameter       values
$(\alpha^\star,\beta^\star)$ and  stop it successively  at the hitting
times $T_n$, with $n  \in \{10^3 k ;1 \le k \le  10\}$. For each value
of $n$, the likelihood is computed as detailed in the previous section
and we  compute the  MLE $(\hat  \alpha_n,\hat  \beta_n)$ through
numerical optimisation of this likelihood.  The likelihood optimisation procedure is performed according to the “L-BFGS-B” method of \cite{ByrdEtAl}.
It is worth  mentioning that the length of the random  walk is not $n$
but rather $T_n$, a quantity that  is much larger in practice, see e.g
Section 5.2 in~\cite{Comets_etal}.
Figure~\ref{fig:boxplots} shows the boxplots of the MLE obtained
from $M=$100 iterations of the procedure and increasing values of $n$. The
red horizontal dotted line shows the true parameter value. 
As expected, the MLE converges to the true value as $n$ increases. 

\begin{figure}[h]
  \centering
  \includegraphics[height=4cm,width=\textwidth]{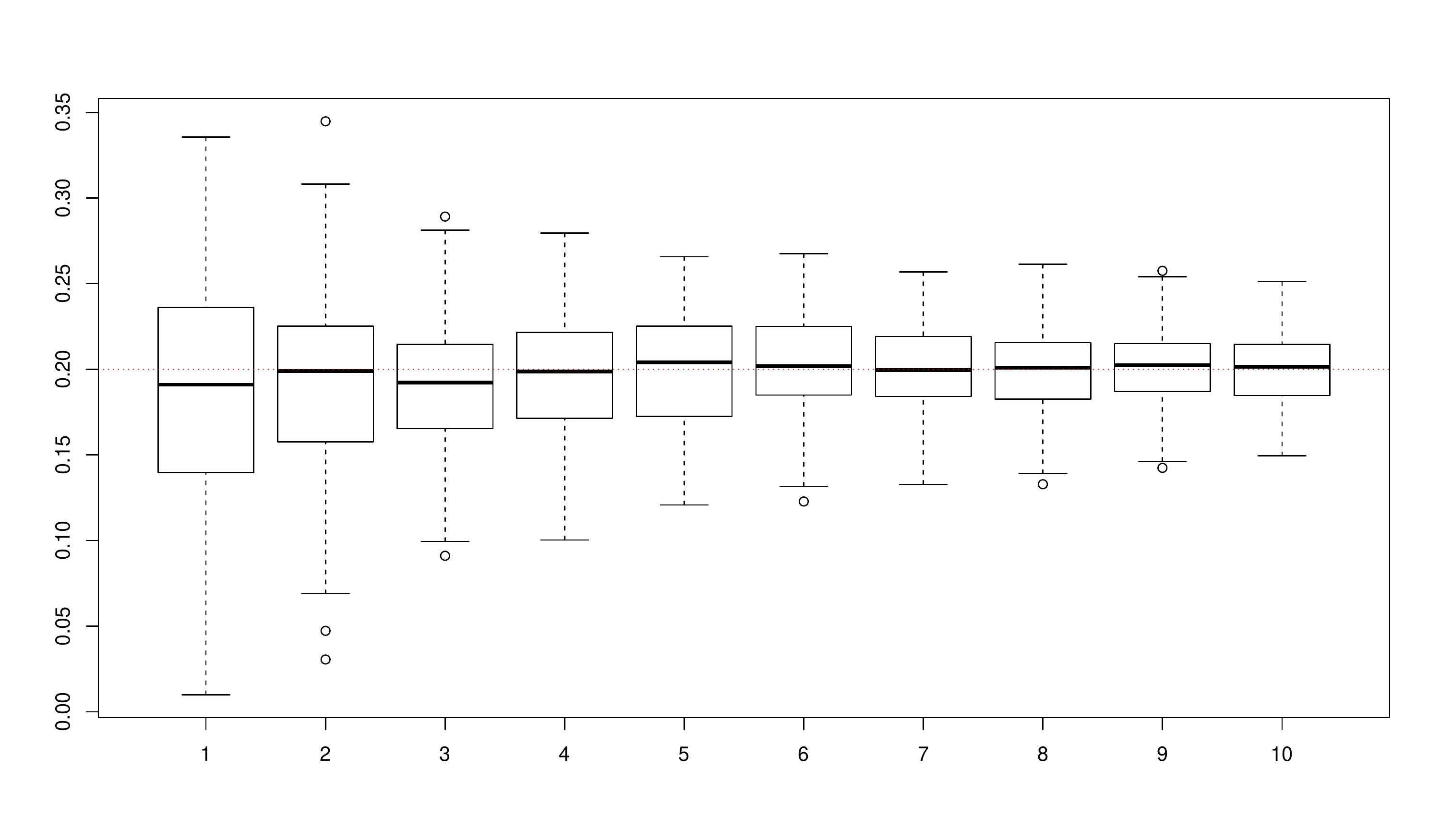}
  \includegraphics[height=4cm,width=\textwidth]{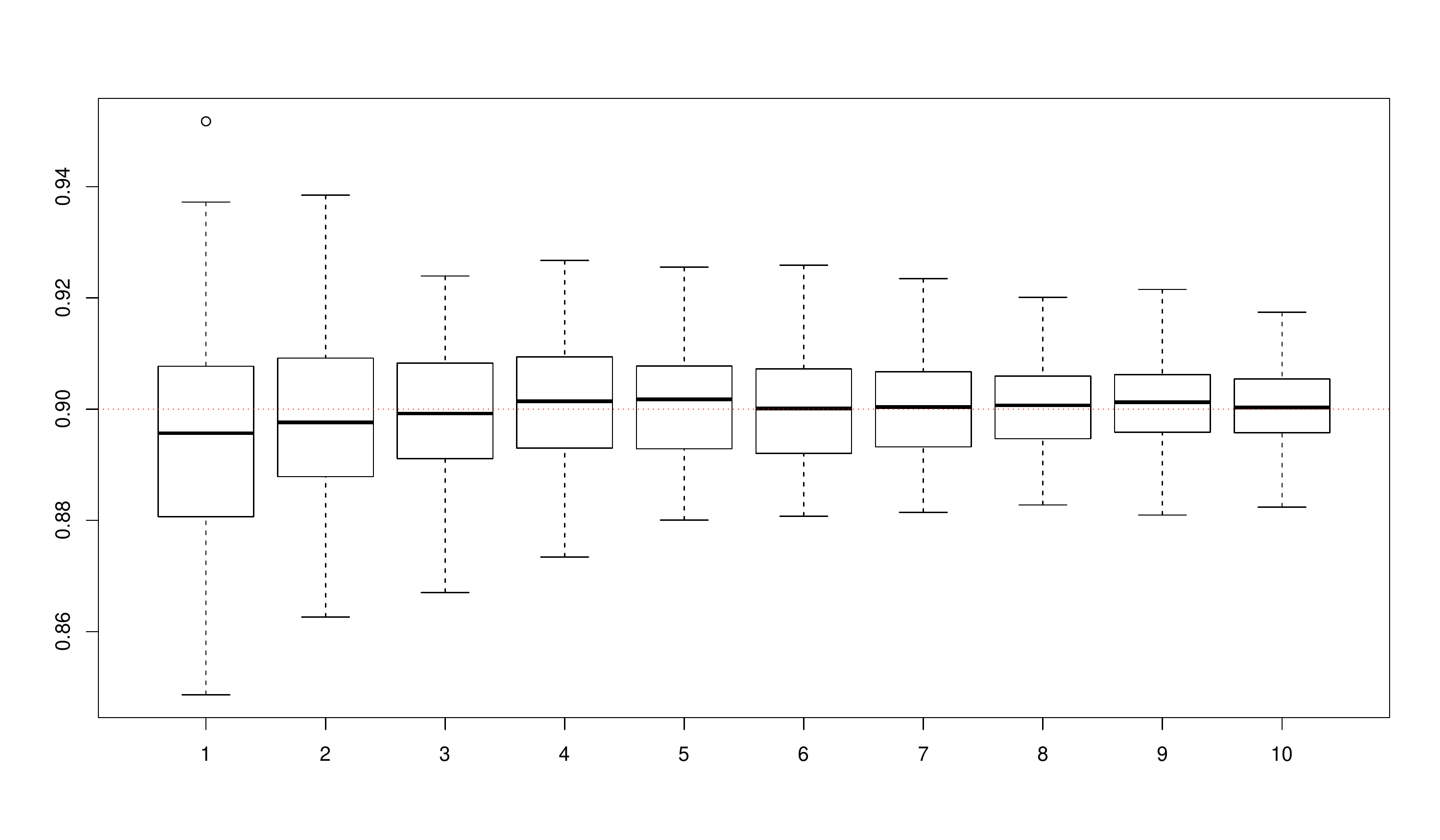}
  \caption{Boxplots of MLE obtained from $M=$100
    iterations and  for values $n$ ranging  in $\{10^3 k ;1  \le k \le
    10\}$ ($x$-axis indicates the value $k$). First and second panel display estimation of
    $\alpha^\star$ and $\beta^\star$, respectively. The true values are indicated by horizontal lines.}
  \label{fig:boxplots}
\end{figure}

We  further  explore  estimation  of the  asymptotic  covariance  matrix
through    the    Hessian     of    the    log-likelihood    according
to~\eqref{eq:estim_Fisher}. Note that  the true value $I(\ts)^{-1}$ is
unknown as there is no constructive form of the Fisher information for
this  model. However,  this  true  value may  be  approximated by  the
empirical   covariance  matrix   obtained  from   running   the  above
experiment with  $M=$100 iterations. Figure~\ref{fig:variances}  shows the
boxplots of the entries of the opposite normalized Hessian of the log-likelihood at
the estimated parameter value, namely 
\[
\hat \Sigma_n := - \frac 1 n \partial_\t^2 \ell_n(\hat \theta),
\]
 obtained  by iterating the  procedure $M=$100  times. The
red  horizontal  dotted  line   does  not  represent  the  entries  of
$I(\ts)^{-1}$ (which remain unknown even if $\ts$ is known) but rather
the entries of the empirical covariance estimate matrix
\[
\widehat {Cov}(\hat \theta_n) :=\frac 1 M \sum_{i=1}^M \Big(\hat \theta_n^{(i)}-\frac 1 M \sum_{i=1}^M \hat
\theta_n^{(i)} \Big) ^\intercal \Big(\hat \theta_n^{(i)}-\frac 1 M \sum_{i=1}^M \hat
\theta_n^{(i)} \Big) ,
\]
where  $\hat  \theta_n^{(i)}$  is  the estimator  obtained  at  $i$-th
iteration.  We  choose  the  most  accurate  estimator  obtained  with
$n=10,000$.  The  results  obtained  are quite  good.

\begin{figure}[h]
      \centering
  \includegraphics[height=8cm,width=\textwidth]{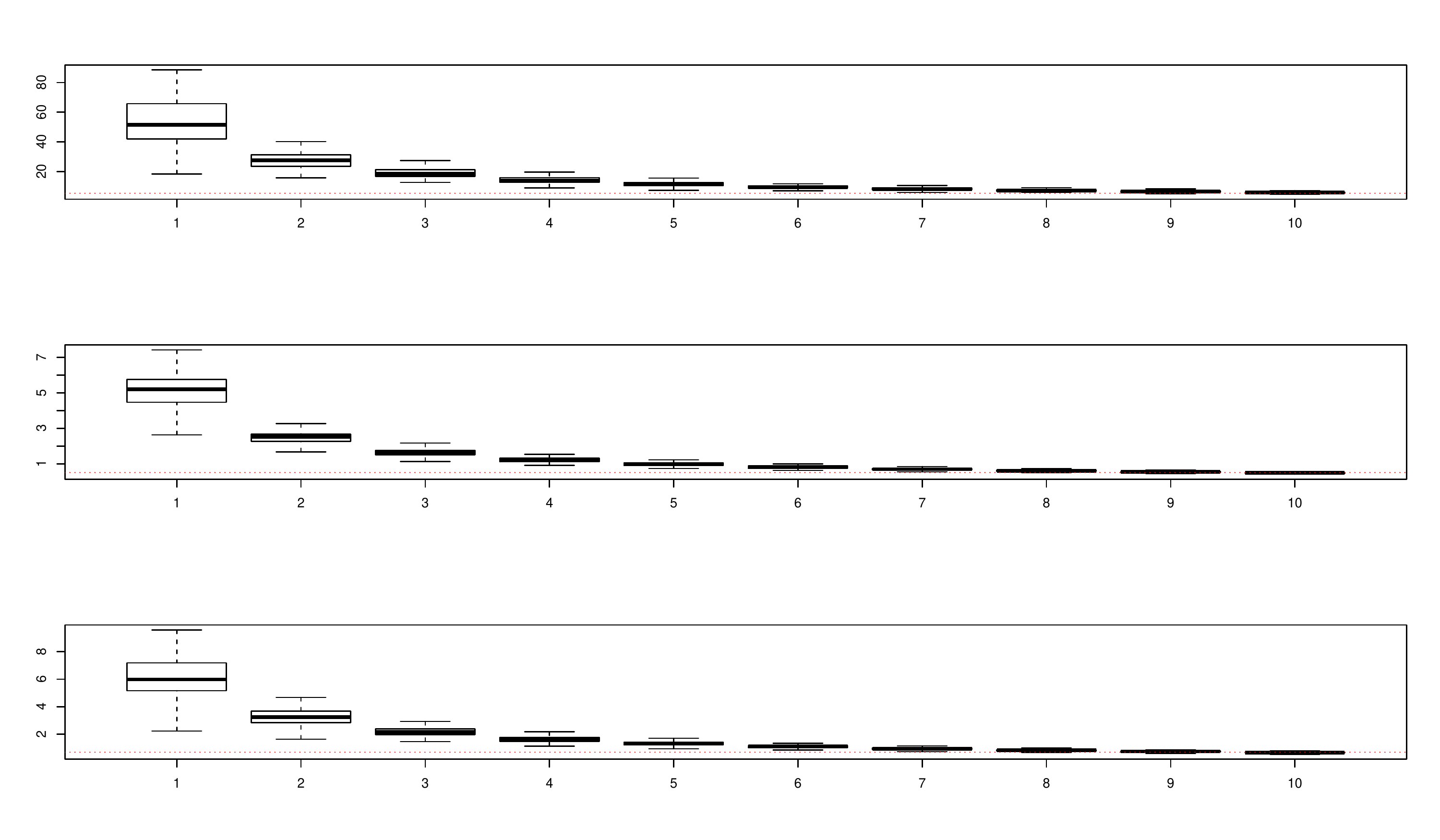}
      \caption{Boxplots of the entries of $\hat \Sigma_n$ obtained from $M=100$ iterations and for values $n$ ranging in $\{10^3 k ;1 \le k \le
        10\}$ ($x$-axis indicates the value $k$). From top to bottom:
        second derivative with  respect to $\alpha$, second derivative
        with respect to $\beta$  and second derivative with respect to
        $\alpha$ and $\beta$. The red dotted line is the empirical estimate
        of the covariance matrix entries 
        obtained  from  $M=100$ iterations  for the  largest value
        $n=10,000$. From top to bottom: $\widehat {Var}(\hat \alpha_n), \widehat{Var}(\hat
        \beta_n), \widehat{Cov}(\hat \alpha_n,\hat \beta_n)$.}
  \label{fig:variances}
    \end{figure}

To conclude  this section, we consider the  construction of confidence
regions for $(\alpha^\star,\beta^\star)$. The  asymptotic
normality  of   the  estimator  $\hat  \theta_n$   together  with  the
estimation  of  the asymptotic  variance  $I(\ts)^{-1}$  leads to  the
following confidence region 
\[
\mathcal{R}_{\gamma,n}  :=\{\theta  \in  \Theta  ,  n  (\hat  \theta_n
-\t)^\intercal \hat \Sigma_n 
(\hat \theta_n -\t) \le \chi_{1-\gamma} \},
\]
where $1-\gamma$  is the asymptotic  confidence level and  $\chi_z$ is
the $z$-th quantile  of the chi-square distribution with  2 degrees of
freedom. Table~\ref{tab:IC} presents  the empirical coverages obtained
from  these confidence  regions $\mathcal{R}_{\gamma,n}$  with $M=100$
iterations and for
$\gamma\in\{0.01,0.05,0.1\}$ and $n$ ranging in $\{10^3 k ;1 \le k \le
        10\}$.  For  the  values  $n\le  6,000$ we  observe  that  the
        confidence regions are too wide. However, for the large values
        $n\ge 9,000$ the empirical coverages are quite good.

\begin{table}[ht]
\centering
\begin{tabular}{|c|ccc|}
  \hline
 $n$ & 0.01& 0.05 & 0.1 \\ 
  \hline
1000 & 1.00 & 1.00 & 1.00 \\ 
  2000 & 1.00 & 1.00 & 1.00 \\ 
  3000 & 1.00 & 1.00 & 1.00 \\ 
  4000 & 1.00 & 1.00 & 1.00 \\ 
  5000 & 1.00 & 1.00 & 0.99 \\ 
  6000 & 1.00 & 0.99 & 0.98 \\ 
  7000 & 1.00 & 0.98 & 0.97 \\ 
  8000 & 0.99 & 0.98 & 0.95 \\ 
  9000 & 0.98 & 0.97 & 0.96 \\ 
  10000 & 0.99 & 0.95 & 0.92 \\ 
   \hline
\end{tabular}
\caption{Empirical coverage of $(1-\gamma)$ asymptotic level confidence
  regions, with $\gamma\in \{0.01,0.05,0.1\}$.} 
\label{tab:IC}
\end{table}


\section{Proofs}\label{sec:proofs} 
\subsection{Properties of the underlying HMM}\label{sec:propr}
In this section, we investigate the properties of the bivariate process
$\{(\breve \w_k,Z_k)\}_{k \ge 0}$, namely we  show that it is positive Harris
recurrent and we exhibit its invariant distribution. 
Let us first define $\breve R$ for the time reversed environment $\breve \bw$ similarly as $R$ from Equation~\eqref{eq:R} by 
\begin{equation}
  \label{eq:breve_R}
 \breve R=(1+\tilde \w_{-1}+\tilde{\w}_{-1}\tilde\w_{-2} +\dots).
\end{equation}
We first remark that Condition~\eqref{eq:ballistic} writes the same for $\bw$ and $\breve \bw$ so that environment
$\breve \bw$ is ballistic under Assumption~\ref{hyp:ballistic} and thus $\EEt(\breve R)<+\infty$.
Moreover,  under Assumptions~\ref{hyp:ballistic} and~\ref{hyp:noyau}, we obtain the following uniform ballistic condition on
  the time reversed environment
\begin{equation}
\label{eq:ballis_extended}
 1  \le \inf_{a\in S}  \EEt_a(\breve R)\le \sup_{a\in  S} \EEt_a(\breve R)
 \le c_+<\infty , 
\end{equation}
for some positive and finite constant $c_+$. 
Indeed, the lower bound follows from $\breve R\ge 1$, by definition of
$R$. Now, for any $a \in (0,1)$, we let $\tilde a
=(1-a)/a$. The upper bound is obtained through 
\begin{align*}
\EEt_a(\breve R) &=1+\EEt_a [\tilde\w_{-1}\EEt_{\w_{-1}}(\breve R) ]
=1+\int_S \tilde b\EEt_b (\breve R ) \breve q_\t(a,b) db\\
&\leq 1+\frac{(1-\eps)\sigma_+}{\eps\sigma_-}\EEt (\breve R),
\end{align*}
where the first equality above is the strong Markov property and the inequality
 uses both~\eqref{eq:S_subset} and the lower bound~\eqref{eq:stat_bound} on the stationary
distribution $\mu_\t$.

The following proposition states the recurrence result on the Markov chain $\{(\breve \w_k,Z_k)\}_{k\ge 0}$ and gives an expression for the density $\pi_\t$ of the corresponding invariant distribution.

\begin{prop}
\label{prop:biv_Harris}
Under Assumptions~\ref{hyp:ballistic} and~\ref{hyp:noyau},  the  Markov chain  $\{(\breve\w_k,Z_k)\}_{k\ge 0}$ whose transition kernel is given by~\eqref{eq:bikernel} is positive Harris recurrent and aperiodic with  invariant density distribution $\pi_\t$ given by 
\[
\forall (a,x)\in S\times \N, \quad
\pi_\t(a,x) =\mu_\t(a)\EEt_a(R^{-1}(1-R^{-1})^x). 
\]
\end{prop}

\begin{proof}
Note that $\pi_\t$ is indeed a density. 
We first prove that it is the density of an invariant distribution. Thus we want to establish that for any $(b,y)\in S\times \N$, we have
\begin{equation}
  \pi_\t(b,y) = \sum_{x\in \N}\int_{S} \pi_\t(a,x)\Pi_\t((a,x),(b,y))da .
  \label{eq:invariance}
\end{equation}
We  start by considering the right-hand side of the above equation where we input the expressions for density $\pi_\t$ and kernel $\Pi_\t$. We let 
\begin{align*}
T & = \sum_{x\in \N}\int_{S} \pi_\t(a,x)\Pi_\t((a,x),(b,y))da \\
 &=    \sum_{x\in     \N}\int_{ S}    \mu_\t(a)\EEt_a[R^{-1}(1-R^{-1})^x]
 \binom{x+y}{x} \breve q_\t (a,b) b^{x+1}(1-b)^y da . 
\end{align*}
From the definition of $\breve q_\theta$ and using  Fubini's theorem for positive functions, we get 
\begin{align*}
T &= \mu_\t(b) \int_{S} q_\t(b,a)(1-b)^y \sum_{x\in \N} \binom{x+y}{x} \EEt_a[R^{-1}(1-R^{-1})^x] b^{x+1} da \\
&= \mu_\t(b) \int_{S} q_\t(b,a)(1-b)^y \EEt_a\left[ \frac {R^{-1}b} {[1-b(1-R^{-1})]^{y+1}} \right] da\\
&= \mu_\t(b) \int_{ S} q_\t(b,a) \EEt_a\left[ \frac{1}{1+ \tilde b
    R}
  \times \left( \frac{1-b}{1-b+bR^{-1}} \right)^y \right] da.
\end{align*}
Now,  applying  Markov's property  and  the  definition  of the  shift
operator, we obtain 
\begin{align*}
T &=\mu_\t(b)\EEt_b\left (\EEt_{\w_1}\left[ \frac{1}{1+ \tilde b
    R} \times
    \left( \frac{1-b}{1-b+bR^{-1}} \right)^y \right]\right )\\
&=\mu_\t(b)\EEt_b\left   (\EEt_{b}\left[  \frac{1}{1+ \tilde b
    R}    \times \left( \frac{\tilde  b R}{1+\tilde b R} \right)^y
    \circ \tau^1 \Big|\F_1 \right]\right )\\
&=\mu_\t(b)\EEt_b\left  (\EEt_{b}\left[  \frac{1}{1+\tilde b+\tilde  b
      \w_1+\ldots}   \times  \left(   \frac{\tilde   b+\tilde  b\tilde
        \w_1+\ldots}{1+\tilde b+\tilde  b\tilde \w_1+\ldots} \right)^y
    \Big| \F_1 \right]\right )\\
&=\mu_\t(b)\EEt_b(R^{-1}(1-R^{-1})^y).
\end{align*}
This concludes the validity of~\eqref{eq:invariance}.

As the  marginal process $\{\w_k\}_{k\ge 0}$ is  aperiodic, this is
also the case for the time reversed marginal process $\{\breve \w_k\}_{k\ge 0}$ and for the bivariate process
$\{(\breve \w_k,Z_k)\}_{k\ge 0}$. 
Following Theorem~9.1.8 in \cite{Meyn_Tweedie}, we want to prove that the Markov chain $\{(\breve \w_k,Z_k)\}_{k\geq 0}$ is $\psi$-irreducible for some probability measure $\psi$ and that there exists a petite set $C\in S\times \N$ and a function $V:S\times \N\to \R_+,$ such that
\begin{enumerate}
  \item $\Delta V(a,x):=\Pi_\t V(a,x)-V(a,x)\leq 0,\quad \forall (a,x)\notin C;$
  \item $\forall N \in \N,\ V_N:= \{ (a,x)\in S\times \N; \ V(a,x)\leq N\}$ is a petite set.
\end{enumerate}
For all $B\in {\cal {B}}(S\times \N)$ and $i\in\{1,2\}$ let $pr_i(B)$ be the projection of $B$ onto $S$ when $i=1$ and
onto $\N$ when $i=2$. We also let $T_B$ be the first hitting time of the set $B$ by the chain  $\{(\breve \w_k,Z_k)\}_{k\geq 0}$ . 
Thanks to Assumptions~\ref{hyp:noyau} and~\eqref{eq:S_subset}, we can write 
   \begin{align*}
  \Pt_{(a,x)}(T_B<\infty)&\geq\int_{pr_1(B)} \sum_{y\in pr_2(B)}\binom{x+y}{x}b^{x+1}(1-b)^y \breve q_\t(a,b) db\\
& \geq 
  \frac{\sigma_-^2}{\sigma_+}\eps^{x}\int_{pr_1(B)}\sum_{y\in pr_2(B)}b(1-b)^ydb .
  \end{align*} 
Hence the Markov chain is $\varphi$-irreducible \citep[see Section 4.2
in][]{Meyn_Tweedie},  where the  measure $\varphi$  defined  on ${\cal
  {B}}(S\times \N)$ by 
\[
B\mapsto \varphi(B):=\int_{pr_1(B)}\sum_{y\in pr_2(B)}b(1-b)^ydb
\]
 is a probability measure. From Proposition 4.2.2 in \cite{Meyn_Tweedie}, the chain is also $\psi$-irreducible. 
 Thanks to Assumption~\ref{hyp:noyau} again, we can easily see that 
  for all $ N\in\N$, the set  $S\times \sem{1,N}$  is a small set and as a consequence a petite set. Indeed, for any $(a,x)\in S\times \sem{1,N}$  we have 
\begin{equation*}
  \Pi_\t((a,x),(b,y))       \ge      \frac{ \sigma_-^2}{\sigma_+} b^{N+1}(1-b)^y\geq  \frac{ \sigma_-^2}{\sigma_+}\eps^{N} b(1-b)^y.
\end{equation*}
Let 
\[
V(a,x)=x\EEt_a(\breve R)=x\EEt_a(1+ \tilde{\w}_{-1}+\tilde{ \w}_{-1}\tilde{ \w}_{-2}+\dots).
\]
By using~\eqref{eq:ballis_extended},  function   $V$   is
finite. Moreover, we get that if $(a,x)\in V_N,$ then
 $ x\leq N,$ which proves that for all $N\in \N$, the set $V_N$ is a petite set. 
Now, we consider
  \begin{align*}
 \Pi_\t V(a,x)&=\int_{ S} \sum_{y\in \N} y\EEt_b(\breve R)\Pi_\t((a,x); (b,y)) db\\
 &=\int_{ S} \sum_{y\in \N} y\EEt_b(\breve R)\binom{x+y}{x}b^{x+1}(1-b)^y {\breve q_\t(a,b)} db\\
 &=(x+1)\int_{ S} {\breve q_\t(a,b)} \left(\frac {1-b}{b}\right)\ \EEt_b (\breve R) db \\
& =(x+1)\EEt_a [\tilde \w_{-1} \EEt_{\w_{-1}} (\breve R)]\\
 &=(x+1)\EEt_a[\tilde \w_{-1} \EEt_a[\breve R\circ\tau^1|\breve {\F}_1]]\\
&=(x+1)\EEt_a[\tilde \w_{-1}(1+\tilde \w_{-2}+\tilde \w_{-2}\tilde \w_{-3}+\dots )]\\
& =(x+1)\EEt_a(\breve R-1).
 \end{align*} 
 Note also that $(c_+-1)\ge \EEt_a(\breve R-1)>0$. As a consequence, for all $(a,x)\notin S\times \sem{0,c_+-1}$ we have $ \Pi_\t V(a,x)\leq V(a,x)$.
This concludes the proof of the proposition.
\end{proof}

\subsection{Proof of consistency}
\label{sec:cons_proof}
Consistency of the maximum likelihood  estimator is given by Theorem 1
in  \cite{DMR_04}  for  the  observations generated  under  stationary
distribution and then extended by Theorem 5 in the same
reference for  a general initial  distribution case. Both  results are
established under some assumptions that we  now investigate
in  our context.  Note that  our process  is not  stationary  since it
starts from $(\w_0,Z_0)\sim \mu_{\ts} \otimes \delta_0$. Thus, we rely
on Theorem 5 from \cite{DMR_04} to establish the properties of our estimator.
We  show  that  our  assumptions   on  the  RWRE  ensure  the  general
assumptions on the autoregressive process with Markov regime needed to
establish the consistency of the MLE~\citep[Assumptions (A1) to (A5) in][]{DMR_04}.

First, Assumption~\ref{hyp:noyau} is sufficient to ensure Assumption~(A1) from~\cite{DMR_04}. Indeed, 
 Assumption~\ref{hyp:noyau} implies Inequalities~\eqref{eq:breve_bound}
which correspond exactly to part (a) of (A1) on transition $\breve q_\theta$. Moreover, statement (b) of (A1) writes in our case as 
\begin{equation*}
\forall (x,y)\in \N^2, \quad 
 \int_{S} g_a(x,y)da = \binom{x+y}{x} \int_{ S} a^{x+1}(1-a)^y da 
\end{equation*}
positive and finite, which is automatically satisfied here.  \\


 Assumption~(A2)  from~\cite{DMR_04}   requires  that  the  transition
 kernel density $\Pi_\t$ of the Markov chain $\{(\breve \w_k,Z_k)\}_{k\ge 0}$
 defined by~\eqref{eq:bikernel} 
 is     positive     Harris     recurrent    and     aperiodic.     In
 Proposition~\ref{prop:biv_Harris}, we  proved that this  is satisfied
 as  soon  as this  is  the case  for  the  environment kernel  $q_\t$
 (namely Assumption~\ref{hyp:noyau}) and under the ballistic assumption~\ref{hyp:ballistic} on the
 RWRE. 
Let us recall that $\bPt$ and $\bEt$ are the probability and
expectation induced  when considering the  chain $\{(\breve \w_k,Z_k)\}_{k\ge
  0}$  under its  stationary distribution  $\pi_\t$.

With the ballistic condition, we obtain  Assumption~(A3) from~\cite{DMR_04}, as stated in the following proposition. 

\begin{prop}
 Under Assumptions~\ref{hyp:ballistic} and~\ref{hyp:noyau}, we have 
\[
\sup_{(x,y)\in  \N^2}\sup_{a\in  S}  g_a(x,y)  <+\infty \text{  and  }
\bEt[\log \int_{ S} g_a(Z_0,Z_1)da]<+\infty . 
\]
\end{prop}

\begin{proof}
The first condition is satisfied according to the definition of $g$ given in~\eqref{eq:emission}. Moreover, we have 
\[
\log  \int_{S}  g_a(Z_0,Z_1)da  = \log\binom{Z_0+Z_1}{Z_0}  +\log
\int_{ S} a^{Z_0+1}(1-a)^{Z_1}da.
\]
Relying on Stirling's approximation, we have 
\[
\log\binom{Z_0+Z_1}{Z_0} = Z_0\log\left(1+\frac{Z_1}{Z_0}\right) +
Z_1\log\left(1+\frac{Z_0}{Z_1}\right) + O_{P}(\log(Z_0+Z_1)),
\] 
where $O_{P}(1)$ stands for a sequence that is bounded in probability.
Thus we can write
\[
\log\binom{Z_0+Z_1}{Z_0} \le Z_0 +Z_1 + O_P(\log(Z_0+Z_1)).
\]
Moreover,  under  assumption~\eqref{eq:S_subset}, we have 
\begin{multline*}
|S| \times [ (Z_0+1)\log \eps+Z_1\log\eps] 
\le \log \int_{ S} a^{Z_0+1}(1-a)^{Z_1}da \\
\le (Z_0+1)\log (1-\eps ) + Z_1\log(1-\eps),
\end{multline*}
where $|S|$ denotes  either the Lebesgue measure  of $S$ or its cardinality  when $S$ is discrete.
As a conclusion, as  soon as $\bEt(Z_0)<+\infty$, the second statement
in the proposition is satisfied.
Now, from the definition of $\pi_\t$ given in Proposition~\ref{prop:biv_Harris}, we get 
\begin{align*}
  \bEt(Z_0)   &=   \sum_{x\in    \N}   \int_{   S}   x\mu_\t(a)
  \EEt_a(R^{-1}(1-R^{-1})^x) da \\
&= \int_{ S} \mu_\t(a) \EEt_a(R-1 ) da\\
& = \EEt(R )-1,
\end{align*}
which is finite thanks to~\ref{hyp:ballistic}.
\end{proof}

Assumption~\ref{hyp:regular} on $q_\theta$  is sufficient to ensure (A4) from~\cite{DMR_04} on $\breve q_\theta$.

Now, we let $\bPtZ$ denote the 
marginal of the distribution  $\bPt$ on the set $\N^\N$ (corresponding
to the  second marginal).  In order to  ensure identifiability  of the
model    \citep[Assumption   (A5)    in   ][]{DMR_04},    we   require
identifiability  of  the  parameter   from  the  distribution  of  the
environment (Assumption~\ref{hyp:ident} in our work). 
\begin{lemm}
  Under Assumption~\ref{hyp:ident}, the autoregressive process with Markov regime has identifiable parameter, i.e.
\[
\forall \t,\t' \in \Theta, \quad 
    \theta = \theta' \iff \bPtZ = \overline{\mathbb{P}}^{\theta',Z}.
\] 
\end{lemm}

\begin{proof}
We prove that $\theta$ is uniquely defined from $\bPtZ$. The knowledge of the distribution $\bPtZ$ means that for any $n\in \N$, any sequence $z_0,\dots,z_n\in \N^{n+1}$, we know the quantity 
\begin{multline*}
\bPt((Z_0,\dots,Z_n)=(z_0,\dots,z_n)) \\
= \int_{S}\dots\int_{S}  \pi_\t(a_0,z_0)\prod_{i=1}^{n}\breve q_\t (a_{i-1},a_i)\prod_{i=1}^n g_{a_i}(z_{i-1},z_i) da_0\dots da_n.
\end{multline*}
Since $g$ does not depend on $\theta$ and is positive, if we assume that $\bPtZ = \overline{\mathbb{P}}^{\theta',Z}$ we obtain from the above expression that 
\[
 \pi_\t(a_0,z_0)\prod_{i=1}^{n}\breve q_\t (a_{i-1},a_i) =  \pi_{\theta'}(a_0,z_0)\prod_{i=1}^{n} \breve q_{\theta'} (a_{i-1},a_i) , 
\]
almost surely (w.r.t. the underlying measure on $S^{n+1}$). 
Noting that Assumption~\ref{hyp:ident} can be formulated on $q_\theta$ or on $\breve q_\theta$ equivalently, this implies $\t=\theta'$. 
\end{proof}

 Now a  direct application from
  Theorem~5 in  \cite{DMR_04} combined with  our previous developments
  establishes that under  Assumptions~\ref{hyp:noyau}  to~\ref{hyp:ident}, the  maximum
  likelihood estimator $\hat \theta_n$ converges $\Pstar$-almost surely to the true parameter
  value $\ts$ as $n$ tends to infinity.

\subsection{Proof of asymptotic normality}
\label{sec:AN_proof}

Applying  Theorem~6 from  \cite{DMR_04} and  using that  in  our case,
their   assumptions   (A6)   to   (A8)   are   satisfied  for $\breve q_\theta$ under   our
Assumption~\ref{hyp:regular2}, we  obtain the weak  convergence of the
conditional  score to  a  Gaussian distribution,  as  soon as  the
  asymptotic variance  is defined, which  means as soon as  the Fisher
  information matrix is invertible.



\end{document}